\documentclass[12pt,oneside,reqno]{amsart}
\usepackage[margin=1in]{geometry}
\usepackage{color}
\usepackage{esint,amssymb}
\usepackage{graphicx}
\usepackage{MnSymbol}
\usepackage{mathtools}
\usepackage[colorlinks=true, pdfstartview=FitV, linkcolor=blue, citecolor=blue, urlcolor=blue,pagebackref=false]{hyperref}
\usepackage{microtype}
\usepackage{amsmath}
\usepackage{relsize}
\usepackage{cancel}
\usepackage{xifthen}
\usepackage{verbatim}
\usepackage{amsxtra}
\definecolor{darkgreen}{rgb}{0,0.5,0}
\definecolor{darkblue}{rgb}{0,0,0.7}
\definecolor{darkred}{rgb}{0.9,0.1,0.1}

\newtheorem{theorem}{Theorem}
\newtheorem{proposition}[theorem]{Proposition}
\newtheorem{lemma}[theorem]{Lemma}

\theoremstyle{definition}
\newtheorem{remark}[theorem]{Remark}

\newcommand{\cref}[1]{Corollary~\ref{c.#1}}

\newcommand{\eref}[1]{(\ref{e.#1})}

\numberwithin{equation}{section}
\numberwithin{theorem}{section}

\newcommand{\Z}{\mathbb{Z}}

\newcommand{\R}{\mathbb{R}}

\newcommand{\T}{\mathbb{T}}

\newcommand{\ep}{\varepsilon}

\newcommand{\test}[1][]{%
\ifthenelse{\equal{#1}{}}{omitted}{given}%
}

\renewcommand{\d}[1]{\ensuremath{\operatorname{d}\!{#1}}}
\DeclarePairedDelimiter{\norm}{\lVert}{\rVert}

\DeclareMathOperator{\sign}{sign}

\newcommand{\jap}[1]{\left\langle {#1} \right\rangle}

\renewcommand{\bar}{\overline}

\renewcommand{\part}{\partial}

\usepackage{dsfont}

\def\f {\frac}
\def\rd {\partial}
\def\ls {\lesssim}
\def\de {\delta}
\def\i {\infty}
\def\alp {\alpha}
\def\bt {\beta}

\def\ep {\epsilon}

\newcommand{\ud}{\mathrm{d}}

\def\la {\langle}
\def\ra {\rangle}

\def \vp {\varphi}

\begin{document}
\title[Phase mixing in a confining potential]{Phase mixing for solutions to \\ 1D transport equation in a confining potential}

\author[S. Chaturvedi]{Sanchit Chaturvedi}
\address[Sanchit Chaturvedi]{Department of Mathematics, Stanford University, 450 Jane Stanford Way, Bldg 380, Stanford, CA 94305, USA}
\email{sanchat@stanford.edu}
\author[Jonathan Luk]{Jonathan Luk}
\address[Jonathan Luk]{Department of Mathematics, Stanford University, 450 Jane Stanford Way, Bldg 380, Stanford, CA 94305, USA}
\email{jluk@stanford.edu}
\subjclass[2010]{}
\date{\today}
\maketitle

\vspace{-2ex}
\begin{center}{\it\large Dedicated to the memory of Robert Glassey}
\end{center}

\begin{abstract}
Consider the linear transport equation in $1$D under an external confining potential $\Phi$:
\begin{equation*}
\rd_t f + v \rd_x f - \rd_x \Phi  \rd_v f = 0.
\end{equation*}
For $\Phi = \f {x^2}2 + \f{\ep x^4}2$ (with $\ep >0$ small), we prove
phase mixing and quantitative decay estimates for $\rd_t \varphi := - \Delta^{-1} \int_{\R} \rd_t f \, \ud v$, with an inverse polynomial decay rate $O(\la t\ra^{-2})$. In the proof, we develop a commuting vector field approach, suitably adapted to this setting. We will explain why we hope this is relevant for the nonlinear stability of the zero solution for the Vlasov--Poisson system in $1$D under the external potential $\Phi$.
\end{abstract}
\section{Introduction}
Consider the linear transport equation in $1$D 
\begin{equation}\label{eq:transport}
\rd_t f + v \rd_x f - \rd_x \Phi  \rd_v f = 0,
\end{equation}
for an unknown function $f: [0,\infty) \times \mathbb R_x \times \mathbb R_v \to \mathbb R_{\geq 0}$
with a smooth external confining potential $\Phi:\mathbb R \to \mathbb R$. 

The following is the main result of this note:
\begin{theorem}\label{thm:main}
Let $\ep>0$ and $\Phi(x) = \f {x^2}2 + \f{\ep x^4}2$. Consider the unique solution $f$ to \eqref{eq:transport} with initial data $f \restriction_{t=0} = f_0$ such that 
\begin{itemize}
\item $f_0: \mathbb R_x \times \mathbb R_v \to \mathbb R_{\geq 0}$ is smooth, and
\item there exists $c_s >0$ such that $\mathrm{supp}(f_0) \subseteq \{ (x,v):  c_s \leq \f{v^2}2 + \Phi(x) \leq c_s^{-1} \}$.
\end{itemize}

Then, for $\ep$ sufficiently small, there exists $C>0$ depending on $\ep$ and $c_s$ such that the following estimate holds:
$$\sup_{x\in \R} |\rd_t \varphi|(t,x) \leq C \la t\ra^{-2} \sup_{(x,v) \in \mathbb R\times \mathbb R} \sum_{|\alp|+ |\bt| \leq 2} |\rd_x^\alp \rd_v^\bt f_0|(x,v),$$
where $\varphi$ is defined by
\begin{equation}\label{eq:phi}
\rd_{xx}^2\varphi(t,x) =  \int_{\R} f(t,x,v)\, \ud v,\quad \varphi(t,0) = \rd_x \varphi(t,0) = 0.
\end{equation}
\end{theorem}

A few remarks of the theorem are in order.

\begin{remark}[Nonlinear Vlasov--Poisson system]
The reason that we are particularly concerned with $\rd_t \varphi$ is that it appears to be the quantity relevant for the stability of the
zero solution for the nonlinear Vlasov--Poisson system in $1$D; see Section~\ref{sec:VP}.

It should be noted that $\varphi$ itself is not expect to decay to $0$ (since $\int_{\R} f\, \ud v \geq 0$). Thus the decay for $\rd_t \varphi$ can be viewed as a measure of the rate that $\varphi$ approaches the limit $\lim_{t\to +\infty} \varphi(t,x)$.

\end{remark}

\begin{remark}[Derivatives of $\rd_t\varphi$]
For the applications on the Vlasov--Poisson system, one may also wish to obtain estimates for the derivatives of $\rd_t \varphi$. It is easy to extend our methods to obtain
$$|\rd_x \rd_t\varphi |\ls \la t\ra^{-1},\quad |\rd_x^2 \rd_t \varphi|\ls 1.$$
Notice that these decay rates, at least by themselves, do not seem sufficient for a global nonlinear result.
\end{remark}

\begin{remark}[Phase mixing and the choice of $\Phi$]
The result in Theorem~\ref{thm:main} can be interpreted as a quantitative phase-mixing statement. It is well-known that for 
$$\Phi(x) = \f {x^2}2,$$
the solution to \eqref{eq:transport} does \underline{not} undergo phase mixing (see chapter 3 in \cite{BinTre_2011}). It is therefore important that we added the $\f{\ep x^4}2$ term in the definition of the potential.

On the other hand, there are other choices of $\Phi$ for which analogues of Theorem~\ref{thm:main} hold. We expect that as long as $\Phi$ is even and satisfies the non-degeneracy condition of \cite{pRoS2020}, then a similar decay estimate holds. The particular example we used is only chosen for concreteness.
\end{remark}

\begin{remark}[Method of proof]
It is well-known that the linear transport equation \eqref{eq:transport} can be written in action-angle variables, say $(Q, K)$, in which case \eqref{eq:transport} takes the form
\begin{equation}\label{eq:trivial}
\part_t f - c({K})\part_{{Q}} f=0.
\end{equation}

When $c'(K)$ is bounded away from $0$, phase mixing in the sense that $f$ converges \emph{weakly} to a limit can be obtained after solving \eqref{eq:trivial} with a Fourier series in $Q$; see \cite{pRoS2020}. The point here is that $\varphi$ is a (weighted) integral of $f$ over a region of phase space that is most conveniently defined with respect to the $(x,v)$ (as opposed to the action-angle) variables.

We quantify the strong convergence of $\varphi_t\to 0$ by finding an appropriate commuting vector field $Y$ that is adapted to the action-angle variables. The fact that $\varphi$ is naturally defined as an integral over $v$ in $(x,v)$ coordinates makes it tricky to prove decay using this vector field. Furthermore, we are only able to prove $1/\jap{t}^2$ decay; this is for instance in contrast to the decay of the density for the free transport equation on a torus.
\end{remark}

\medskip

\subsection{Related result}

\subsection*{Linear phase mixing results}
In the particular context of Theorem~\ref{thm:main}, decay of $\rd_t\varphi$, but without a quantitative rate, can be inferred from the work \cite{pRoS2020}.

There are many linear phase mixing result, the simplest setting for this is the linear free transport equation. This is well-known; see for instance notes \cite{villani2010landau} by Villani.

One of the most influential work on phase mixing is the groundbreaking paper \cite{Landau1946} of Landau wherein he proposes a linear mechanism for damping for plasmas that does not involve dispersion or change in entropy. In the case of $\mathbb T^d$, this is even understood in a nonlinear setting; see the section on nonlinear results below. The situation is more subtle in $\mathbb R^d$, see \cite{jBnMcM2020}, \cite{rGjS1994}, \cite{rGjS1995} and \cite{dHKttNfR2020}.

See also \cite{jBfW2020}, \cite{faou2021linear} and \cite{iT2017} for linear results on related models. In particular, we note that \cite{faou2021linear} also rely on action-angle variables in their analysis.

\subsection*{Relation with other phase-mixing problems with integrable underlying dynamics}
As pointed out in \cite{pRoS2020}, phase space mixing is relevant for the dynamics of kinetic models in many physical phenomena from stellar systems and dark matter halos to mixing of relativistic gas surrounding a black hole. See \cite{dominguez2017description} for related discussions on dark matter halos. We also refer the interested reader to \cite{BinTre_2011} for further background and discussions of phase mixing in other models, including the stability of galaxies.

We hope that the present work would also be a model problem and aid in understanding more complicated systems such as those described in \cite{pRoS2020}. One particularly interesting problem is the stability of the Schwarzschild solution to the Einstein--Vlasov system in spherical symmetry. 

\subsection*{Nonlinear phase mixing results}
Nonlinear Landau damping for Vlasov--Poisson on $\T^d$ was first proven in analytic regularity by Mouhot--Villani in their landmark paper  \cite{cMcV2011}. Since then their work has been extended and simplified in \cite{jBnMcM2016} and \cite{eGtNiR2020a}. 

See also other nonlinear results, e.g.~in \cite{jB2017}, \cite{jBnMcM2018}, \cite{chaturvedi2021vlasov}, \cite{faou2016landau}, \cite{dHKttNfR2019}, \cite{bY2016}.

\subsection*{Collisional problems with confining potentials}

Confining potentials for kinetic equations have been well-studied, particularly for collisional models. Linear stability results can be found in \cite{carrapatoso2021special}, \cite{dolbeault2009hypocoercivity}, \cite{dolbeault2015hypocoercivity}, \cite{duan2011hypocoercivity} and \cite{duan2012hypocoercivity}.

In this connection, it would also be of interest to understand how phase mixing effects (studied in the present paper) interact with collisional effects (cf.~\cite{jB2017}, \cite{chaturvedi2021vlasov}, \cite{iT2017}.)

\section{The Vlasov--Poisson system}\label{sec:VP}

The motivation of our result is the Vlasov--Poisson system:
\begin{equation}\label{eq:VP}
\begin{cases}
\part_t f+v\part_x f-(\part_x \Phi+\part_x \vp)\part_v f=0, \\
-\part^2_x \varphi=\int_{\R} f\d v.
\end{cases}
\end{equation}

Note that \eqref{eq:VP} can be rewritten as 
\begin{equation}\label{e.VP}
\part_t f+\{H,f\}=0,
\end{equation}
where $H$ is the Hamiltonian given by 
\begin{equation}\label{eq:nonlin-hamiltonian}
H(x,v)=\frac{v^2}{2}+\Phi(x)+\vp(t,x).
\end{equation}

Notice that $f \equiv 0$ is a solution to \eqref{eq:VP}, and the transport equation \eqref{eq:transport} is 
the linearization of \eqref{eq:VP} near the zero solution.

One cannot hope that the term $\rd_x \varphi$ in the nonlinear term decays as $t\to +\infty$. 
(This can be seen by noting that $\int_{\R} f\d v \geq 0$ pointwise.) At best one can hope 
that $\rd_x \varphi$ converges to some (non-trivial) limiting profile as $t\to +\infty$. For $f$ satisfying 
the linear equation \eqref{eq:transport}, such convergence (without a quantitative rate) has been shown
in \cite{pRoS2020}.

In anticipation of the nonlinear problem, it is important to understand the quantitative convergence. Since
$\rd_x\varphi$ does not converge to $0$, it is natural to understand the decay rate of $\rd_t \rd_x \varphi$.

As a first step to understand \eqref{eq:VP}, we look at the linearized problem \eqref{eq:transport} around the zero solution and prove that we get integrable decay for $\varphi_t$ in the linearized dynamics.

\begin{remark}
Note that the Poisson's equation above reads
$$-\part^2_x \vp=\rho.$$
In particular, $\vp$ is only defined up to a harmonic function, i.e.~a linear function a $x$. 
In Theorem~\ref{thm:main}, we remove this ambiguity by setting $\vp(0) = (\rd_x \vp)(0) = 0$. 
Notice that other normalization, e.g., $\vp(-\infty) = (\rd_x \vp)(-\infty)=0$ would not change the function $\vp_t = \rd_t \vp$.
\end{remark}

\section{The action-angle variables}\label{sec:action}

\subsection{First change of variables}
From now on we will consider the Hamiltonian $$H=\frac{v^2}{2}+\Phi(x).$$ This is the Hamiltonian for the equations \eqref{eq:transport}, which is also \eqref{eq:nonlin-hamiltonian} without the $\varphi$ (the self-interaction term). As an intermediate step to getting the action-angle variables we use the change of coordinates 
\begin{align*}
(t,x,v)\mapsto (t,\chi,H) &\hspace{5 em}\text{when $x>0$}\\
(t,x,v)\mapsto (t,\pi-\chi,H) &\hspace{5 em}\text{when $x\leq 0$},
\end{align*}
where $\chi:=\arcsin\left(\frac{v}{\sqrt{2H}}\right).$\\
First we check if the change of variables is well defined by calculating the Jacobian for $x>0$,
\begin{equation*}
J=\begin{pmatrix}
\part_t t& \part_x t&\part_v t\\
\part_t H& \part_x H&\part_v H\\
\part_t \chi& \part_x \chi&\part_v \chi
\end{pmatrix}
=
\begin{pmatrix}
1&0&0\\
0&\Phi_x&v\\
0&-\frac{v}{2H}\cdot\frac{\Phi_x}{\sqrt{\Phi}}&\frac{\sqrt{\Phi}}{H}
\end{pmatrix}
\end{equation*}
Now 
\begin{align*}
\det(J)&=\frac{\Phi_x}{\sqrt{\Phi}}\frac{(\Phi+v^2/2)}{H}\\
&=\frac{\Phi_x}{\sqrt{\Phi}}.
\end{align*}
Similarly for $x\leq 0$, \begin{align*}
\det(J)&= - \frac{\Phi_x}{\sqrt{\Phi}}.
\end{align*}
Hence, \begin{align*}
\det(J)
&=\sign{x}\frac{\Phi_x}{\sqrt{\Phi}}.
\end{align*}
Next by chain rule and using that $H$ is independent of $t$, we get

\begin{align*}
\part_x&=\sign{x}\part_x \chi\part_{\chi}+\part_x H\part_H\\
&=-\frac{v}{2H}\cdot\frac{\Phi_x}{\sqrt{\Phi}}\part_{\chi}+\Phi_x\part_H,
\end{align*}
\begin{align*}
\part_v&=\sign{x}\part_v \chi\part_{\chi}+\part_v H\part_H\\
&=\frac{\sqrt{\Phi}}{H}\part_{\chi}+v\part_H.
\end{align*}
Pluggin this in \eref{VP}, we get the equation
\begin{equation}\label{e.VP_chi_H_lin}
\part_t f-\sign{x}\frac{\Phi_x}{\sqrt{\Phi}}\part_\chi f=0.
\end{equation}

\subsection{Second change of variables} The coefficient in front of $\part_\chi f$ in \eqref{e.VP_chi_H_lin} depends on both $\chi$ and $H$. To take care of this, we reparametrize $\chi$ (in a manner depending on $H$). More precisely, for a fixed $H$, we define $Q(\chi,H)$ such that 
$$\frac{\d Q}{\d \chi}=\frac{c(H)}{a(\chi,H)},\quad Q(0,H)=0,$$
where $a(\chi,H)=\sign{x}\frac{\Phi_x(x)}{\sqrt{\Phi(x)}}$ such that $x=x(\chi,H).$ To fix $c(H)$, we require that for every $H$, 
\begin{equation}\label{eq:C.def}
2\pi=\int_0^{2\pi}\d Q=c(H)\int_0^{2\pi}\frac{1}{a(\chi,H)}\d \chi.
\end{equation}
Now we define the change of variables, $(\chi,H)\mapsto (Q,K)$ where $K=H$. Then note,
$$a(\chi,H)\partial_\chi=c(H)\part_Q$$
and $$\part_{H}=\part_{K}+\frac{\part Q}{\part H}\part_Q.$$
Thus in these coordinates, we can rewrite \eref{VP_chi_H_lin} as 
\begin{equation}\label{e.VP_Q_H_lin}
\part_t f-c({K})\part_Q f=0.
\end{equation}

Further, the Jacobian is 
\begin{equation*}
\begin{pmatrix}
\part_H K&\part_\chi K\\
\part_H Q& \part_\chi Q
\end{pmatrix}
=
\begin{pmatrix}
1&0\\
\part_H Q&\frac{c(H)}{a(\chi,H)}
\end{pmatrix}.
\end{equation*}
Note that the determinant is $\frac{c(H)}{a(\chi,H)}$. Further, since 
\begin{align*}
a(\chi,H)&=\sign{x}\frac{\Phi_x(x)}{\sqrt{\Phi(x)}}\\
&=\sqrt{2}\frac{1+2\varepsilon x^2}{\sqrt{1+\varepsilon x^2}},
\end{align*}
	we have that $a(\chi,H) \approx 1$ when $x$ is in a compact subset of $\mathbb R$. As a result the determinant is bounded away from zero. For more details see Lemma~\ref{lem:c-prime-bound}.


\section{The commuting vector field}
We first define the vector field $$Y=tc'(H)\partial_Q-\part_K.$$ In this section we prove that this vector field commutes with the transport operator as in \eqref{e.VP_Q_H_lin} and that $|c'(H)|>0.$
\subsection{Commutation property}

The following commutation formula is an easy computation and thus we leave out the details.
\begin{lemma}\label{lem:commutation}
Let $Y = t c'(H) \rd_Q - \rd_{K}$. Then
$$[\rd_t - c(H) \rd_Q, Y] = 0.$$
\end{lemma}

The following is an easy consequence of Lemma~\ref{lem:commutation}:
\begin{lemma}\label{lem:infinity-est}
Let $f$ be a solution to \eqref{e.VP_Q_H_lin} with initial data satisfying assumptions of Theorem~\ref{thm:main}. Then
$$\sup_{(t,Q,K) \in [0,\infty)\times \mathbb T^1 \times [c_s, c_s^{-1}]}\hspace{.1em}\sum_{\ell\leq 2} |Y^{\ell} f|(t,Q,K) \ls \sup_{(x,v) \in \mathbb R\times \mathbb R} \hspace{.1em}\sum_{|\alp|+ |\bt| \leq 2} |\rd_x^\alp \rd_v^\bt f_0|(x,v).$$
\end{lemma}
\begin{proof}
By Lemma~\ref{lem:commutation}, we have that $Y^\ell f$ satisfies the transport equation \eqref{e.VP_Q_H_lin} for any $\ell \in \mathbb N \cap \{0\}$. Hence we get the estimate
$$\sup_{(t,Q,K) \in [0,\infty)\times \mathbb T^1 \times [c_s, c_s^{-1}]}\hspace{.1em}\sum_{\ell\leq 2} |Y^{\ell} f|(t,Q,K)\ls \sup_{(Q,K) \in \mathbb T^1 \times [c_s, c_s^{-1}]}\hspace{.1em}\sum_{\ell \leq 2} |\part_K^{\ell} f_0|(Q,K).$$
Since the change of variables $(Q,K)\to (x,v)$ is well-defined and bounded away from zero, we get the required result.
\end{proof}
\subsection{Positivity of $|c'(K)|$}
We prove that $|c'(K)|$ is uniformly bounded below on the support of $f$. This plays a key role in the next section ensuring phase mixing.
\begin{lemma}\label{lem:c-prime-bound}
For every $c_s <+\infty$, there exists $\ep_0>0$ such that whenever $\ep \in (0,\ep_0]$, there is a small constant $\de >0$ (depending on $c_s$ and $\ep$) such that 
$$\inf_{K \in [c_s, c_s^{-1}]} |c'(K)| = \inf_{H \in [c_s, c_s^{-1}]} |c'(H)| \geq \de.$$
\end{lemma}
\begin{proof}
By definition of $c(H)$, we have
$$\frac{2\pi}{c(H)}=\int_0^{2\pi}\left|\frac{\sqrt{\Phi}}{\Phi'}\right|\d \chi,$$
so that using $\Phi=\frac{x^2}{2}+\frac{\ep}{2} x^4$, we obtain
$$\frac{2\pi}{c(H)}=\int_0^{2\pi}\frac{\sqrt{1+\ep x^2}}{\sqrt{2}(1+2\ep x^2)}\d \chi.$$

Notice that for $H \in [c_s, c_s^{-1}]$, $|x|$ is bounded. It follows that $c(H) \approx 1$. Therefore, to prove strict positivity of $|c'(H)|$, it suffices to prove positivity of $\left|\frac{c'(H)}{c^2(H)}\right|.$ Note that 
\begin{equation}\label{eq:c'/c}
\begin{split}
\frac{-2\pi c'(H)}{c^2(H)}&=\f 1{\sqrt 2}\int_0^{2\pi}\part_H\left(\frac{\sqrt{1+\ep x^2}}{1+2\ep x^2}\right)\d \chi\\
&=\f 1{\sqrt 2}\int_0^{2\pi}\part_H x\left[\frac{\ep x}{\sqrt{1+\ep x^2}(1+2\ep x^2)}-\frac{4\ep x\sqrt{1+\ep x^2}}{(1+2\ep x^2)^2}\right]\d \chi\\
&=\f 1{\sqrt 2}\int_0^{2\pi}\part_H x\left[\frac{-3\ep x-2\ep^2x^3}{\sqrt{1+\ep x^2}(1+2\ep x^2)^2}\right]\d \chi.
\end{split}
\end{equation}

Now we calculate $\part_H x.$ First we use the equation, $H=\frac{v^2}{2}+\Phi(x).$ Precisely, we have 
$$1=v\part_H v+\Phi'(x) \part_H x.$$ Thus 
\begin{equation}\label{eq:dxH}
\part_H x=\frac{1-v\part_H v}{\Phi'}.
\end{equation}
Next we use that $\frac{v}{\sqrt{2H}}=\sin\chi,$
$$0=\frac{\part_H v}{\sqrt{2H}}-\frac{v}{(2H)^{\frac{3}{2}}}.$$
Thus $\part_H v=\frac{v}{2H}.$ Plugging this into \eqref{eq:dxH}, we get that 
\begin{equation}\label{eq:dxH.final}
\part_H x=\frac{1-\frac{v^2}{2H}}{\Phi'}=\frac{\cos^2 \chi}{\Phi'(x)} = \frac{\cos^2\chi}{x+2\ep x^3},
\end{equation}
where in the last equality we used $\Phi=\frac{x^2+\ep x^4}{2}$.

Plugging \eqref{eq:dxH.final} back into \eqref{eq:c'/c}, we get that
\begin{align*}
\frac{-2\pi c'(H)}{c^2(H)}&=\int_0^{2\pi}\frac{\cos^2\chi}{\sqrt{2} x(1+2\ep x^2)}\left[\frac{-x(3\ep +2\ep^2x^2)}{\sqrt{1+\ep x^2}(1+2\ep x^2)^2}\right]\d \chi\\
&=-\int_0^{2\pi}{\cos^2\chi}\left[\frac{(3\ep +2\ep^2x^2)}{\sqrt{2(1+\ep x^2)}(1+2\ep x^2)^3}\right]\d \chi.
\end{align*}
Finally note that since $|x|$ is bounded on the region of interest, after choosing $\ep_0$ sufficiently small, we have
$$\frac{(3\ep +2\ep^2x^2)}{\sqrt{2(1+\ep x^2)}(1+2\ep x^2)^3}\approx \ep,$$ 
and thus $|c'(H)|>\de$. \qedhere
\end{proof}

\section{Decay for $\vp_t$}
In this section we finally prove the decay for $\varphi_t$ (recall Theorem~\ref{thm:main}). 

To keep the notation lean, we will often suppress the explicit dependence on $t$.
\begin{lemma}\label{lem:phi_t}
For $f$ satisfying the assumptions of Theorem~\ref{thm:main}, and $\varphi$ defined as in \eqref{eq:phi}, we have the following formula
$$\vp_t(x')=\int_0^{x'}\int_{\R}v[f(y,v)-f(0,v)]\d v\d y.$$
\end{lemma}
\begin{proof}
By the continuity equation (following directly from \eqref{eq:transport}), we have that $$\rho_t=\int_{\R}v \part_x f\d v.$$
Thus $$-\part^2_x \vp_t=\rho_t=\int_{\R}v \rd_x f\d v.$$
Solving the Laplace's equation (with boundary conditions \eqref{eq:phi}), we get
$$\vp_t(x')=\int_0^{x'}\int_0^y \int_{\R}v\part_x f(z,v)\d v\d z\d y.$$
Integrating by parts in $z$, we get
$$\vp_t(x')=\int_0^{x'}\int_{\R}v[f(y,v)-f(0,v)]\d v\d y.$$
\end{proof}

In view of Lemma~\ref{lem:phi_t}, it suffices to bound $\int_0^{x'}\int_{\R} v f(0,v) \d v\d y$ and $\int_0^{x'}\int_{\R} v f(y,v) \d v\d y$, which will be achieved in the next two subsections respectively.

\subsection{Decay for the term involving $f(0,v)$}
We first prove decay for $\int_{\R} v f(0,v)\d v.$ {Before proving the main estimate in Proposition~\ref{prop:term-x=0}, we first prove a lemma.}

\begin{lemma}\label{lem:Q.at.pi.2}
The level set $\{x = 0\}$ corresponds to the level sets $\{ Q = \f \pi 2 \} \cap \{ Q = -\f \pi 2\} \cup \{(x,v) = (0,0)\}$. 
\end{lemma}
\begin{proof}
First note that level set $\{x = 0\}$ corresponds to the level sets $\{ \chi = \f \pi 2 \} \cap \{ \chi = -\f \pi 2\} \cup \{(x,v) = (0,0)\}$. This is because when $x = 0$, $\Phi(x)= 0 $, and thus by definition (when $v \neq 0$) $\chi:=\arcsin\left(\frac{v}{\sqrt{2H}}\right) = \arcsin (\pm 1) = \pm \f \pi 2$.

It thus remains to show that
\begin{equation}\label{eq:chi.Q.pi.2}
\chi = \pm \f \pi 2 \iff Q = \pm \f \pi 2.
\end{equation}

Fix $H$, then since $a(\chi,H)=\lvert\frac{\Phi_x}{\sqrt{\Phi}}\rvert$ is independent of $v$, we have
$$c(H)\int_0^{\pi} \frac{1}{a(\chi,H)}\d \chi=c(H)\int_{\pi}^{2\pi} \frac{1}{a(\chi,H)}\d \chi.$$
Further, by the evenness of $\Phi$, we have
$$c(H)\int_0^{\pi/2} \frac{1}{a(\chi,H)}\d \chi=c(H)\int_{\pi/2}^{\pi} \frac{1}{a(\chi,H)}\d \chi.$$
Finally, since we have by construction, $$c(H)\int_0^{2\pi} \frac{1}{a(\chi,H)}\d \chi=2\pi,$$ we have that $$Q(\chi=\pi/2,H)=c(H)\int_0^{\pi/2} \frac{1}{a(\chi,H)}\d \chi=\pi/2.$$
Similarly, $Q(\chi=-\pi/2,H)=-\pi/2.$ Combining these, we obtain \eqref{eq:chi.Q.pi.2}. \qedhere 
\end{proof}

\begin{proposition}\label{prop:term-x=0}
For $f$ satisfying the assumptions of Theorem~\ref{thm:main}, we have the following estimate:
$$\Big| \int_{\R} v f(0,v)\d v \Big| \ls \la t \ra^{-2} \sup_{(x,v) \in \mathbb R\times \mathbb R} \hspace{.1em}\sum_{|\alp|+ |\bt| \leq 2} |\rd_x^\alp \rd_v^\bt f_0|(x,v).$$
\end{proposition}
\begin{proof}
The transport equation preserves $L^\i$ bounds so that by the support properties, we obviously have
$$\Big| \int_{\R} v f(0,v)\d v \Big| \ls  \sup_{(x,v) \in \mathbb R\times \mathbb R} \hspace{.1em} | f_0|(x,v).$$
In other words, it suffices to prove the desired bound with $t^{-2}$ instead of $\la t \ra^{-2}$.

Now note that $$\int_{\R} v f(0,v)\d v=\int_0^\infty v [f(0,v)-f(0,-v)]\d v.$$
For clarity of notation, we let $$\bar{f}(t,Q,K) = f(t,x,v).$$ Now writing in the $(K,Q)$ variables, and using Lemma~\ref{lem:Q.at.pi.2} together with the fact that $K = H = \f {v^2}2$ when $x=0$, we have
$$\int_{\R} v f(0,v)\d v=\int_0^\infty v [f(0,v)-f(0,-v)]\d v=\int_0^\infty [\bar f(\pi/2,K)-\bar f(-\pi/2,K)]\d { K}.$$
By the fundamental theorem of calculus, we have
$$\int_0^\infty [\bar f(\pi/2,K)- \bar f(-\pi/2,K)]\d { K}=\int_{-\pi/2}^{\pi/2}\int_0^\infty \part_Q \bar f(Q,K)\d {K}\d Q.$$
Next, the Cauchy--Schwarz inequality implies
\begin{align*}
\int_{-\pi/2}^{\pi/2}\int_0^\infty \part_Q  \bar f(Q,K)\d {K}\d Q&=\sqrt{\pi}\left(\int_{-\pi/2}^{\pi/2}\left(\int_0^\infty \part_Q  \bar f(Q,K)\d {K}\right)^2\d Q\right)^{\frac{1}{2}}\\
&\lesssim \left(\int_{0}^{2\pi}\left(\int_0^\infty \part_Q  \bar f(Q,K)\d {K}\right)^2\d Q\right)^{\frac{1}{2}}.
\end{align*}
Now using Poincare's inequality we get that for any $\ell\geq 2$
\begin{equation}\label{eq:easy.term.after.Poincare}
\left(\int_{0}^{2\pi}\left(\int_0^\infty \part_Q  \bar f(Q,K)\d {K}\right)^2\d Q\right)^{\frac{1}{2}}\lesssim \left(\int_{0}^{2\pi}\left(\int_0^\infty \part^{\ell}_Q  \bar f(Q,K)\d {K}\right)^2\d Q\right)^{\frac{1}{2}}.
\end{equation}

Now take $\ell = 2$. We write $\part_Q=\frac{1}{c'(K)t}(Y+\part_{K})$ so that 
\begin{equation*}
\begin{split}
&\: \left(\int_{0}^{2\pi}\left(\int_0^\infty \part^2_Q  \bar f(Q,K)\d {K}\right)^2\d Q\right)^{\frac{1}{2}} \\
= &\: \left(\int_{0}^{2\pi}\left(\int_0^\infty \frac{1}{|c'(K)|^2 t^2} ( Y^2 \bar f+ 2\part_{K} Y \bar f+ \rd^2_K \bar f) (Q,K)\d {K}\right)^2\d Q\right)^{\frac{1}{2}} \\
\ls &\: \f 1{t^2} \left(\int_{0}^{2\pi}\left(\int_0^\infty   (\sum_{k=0}^2  |Y^k \bar f|) (Q,K)\d {K}\right)^2\d Q\right)^{\frac{1}{2}}.
\end{split}
\end{equation*}
where in the last step we have integrated by parts in $K$ and bounded $\f 1{|c'(K)|}$, $\f{|c''(K)|}{|c'(K)|}$, etc.~using Lemma~\ref{lem:c-prime-bound} and the smoothness of $c$.

Finally, since $f(t,Q,K)$ is non-zero for $c_s\leq K\leq c_s^{-1}$ and $Q\in [0,2\pi]$, we can take supremum in $K$ and $Q$ followed by Lemma~\ref{lem:infinity-est} to get the required result.\qedhere
\end{proof}

\begin{remark}
Notice that since we can take any $\ell \geq 2$ in \eqref{eq:easy.term.after.Poincare}, we can write each $\part_Q=\frac{1}{c'(H)t}(Y+\part_{K})$ and integrate by parts in $K$ many times to show that the term in Proposition~\ref{prop:term-x=0} in fact decays faster than any inverse polynomial (depending on smoothness of $f$)! 

In other words, the decay rate that we obtain in Theorem~\ref{thm:main} is instead limited by the term treated in Proposition~\ref{prop:bulk-term} below.
\end{remark}

\subsection{Decay for the term involving $f(y,v)$} 

{We now turn to the other term in Lemma~\ref{lem:phi_t}. Before we obtain the main estimate in Proposition~\ref{prop:bulk-term}, we first prove two simple lemmas.}
\begin{lemma}\label{lem:Jacobian}
Under the change of variables $(x,v)\mapsto (Q,K)$ as in Section~\ref{sec:action}, the volume form transforms as follows:
$$\d v \d x = c(K) \,\d Q\d{K}.$$
\end{lemma}
\begin{proof}
The Jacobian determinant for the change of variables $(x,v)\mapsto (\chi,H)$ is $a(\chi,H)=\lvert\frac{\Phi_x}{\sqrt{\Phi}}\rvert$. Further the Jacobian determinant for the change of variables $(\chi,H)\to (Q,K)$ is $\frac{c(H)}{a(\chi,H)}$ and hence the Jacobian determinant for $(x,v)\to (Q,K)$ is $c(H) = c(K)$. \qedhere
\end{proof}

\begin{lemma}\label{lem:f-to-g}
Let $\bar{f}(Q,K) = f(x,v)$ as above. There exists a function $\bar{g}(Q,K)$ such that
\begin{equation}\label{eq:f-to-g}
\rd_Q^2 \bar{g} = \rd_Q \bar{f}
\end{equation}
and
\begin{equation}\label{eq:g-bound}
\max_{\ell\leq 2} \sup_{K} \norm{Y^{\ell} \bar{g}}_{L^2_Q} \ls \max_{\ell\leq 2} \sup_{Q,K} |Y^{\ell} \bar{f}|.
\end{equation}
\end{lemma}
\begin{proof}
We use the Fourier series of $f$ in $Q$ to get that $$\bar{f}(Q,K)=\sum_{k=-\infty}^{k=\infty} \widehat{\bar{f}}_k(K)e^{ikQ}.$$
Now we define $$\bar{g}(Q,K):=\sum_{k\in\Z\backslash\{0\}} \frac{1}{ik}\widehat{\bar{f}}_k(K)e^{ikQ}.$$ Then we see that $\partial^2_Q \bar{g}=\partial_Q \bar{f}.$

Using Plancheral's theorem we can easily see that
$$\max_{\ell\leq 2} \sup_{K} \norm{Y^{\ell} \bar{g}}_{L^2_Q} \ls \max_{\ell\leq 2} \sup_{K} \norm{Y^{\ell} \bar{f}}_{L^2_Q}.$$
Finally, the result follows by taking supremum in $Q$. \qedhere
\end{proof}

\begin{proposition}\label{prop:bulk-term}
For $f$ satisfying the assumptions of Theorem~\ref{thm:main}, we have the following estimate:
$$\int_0^{x'}\int_{\R}vf(t,y,v)\d v\d y\ls \jap{t}^{-2} {\sup_{(x,v) \in \mathbb R\times \mathbb R} }\hspace{.1em} \sum_{|\alp|+ |\bt| \leq 2} |\rd_x^\alp \rd_v^\bt f_0|(x,v).$$
\end{proposition}
\begin{proof}
{As in the proof of Proposition~\ref{prop:term-x=0}, boundedness is obvious and thus it suffices to prove an estimate with $\la t\ra^{-2}$ replaced by $t^{-2}$.}

We first note that
$$\int_0^{x'}\int_{\R}vf(y,v)\d v\d y=\int_0^{x'}\int_0^\infty v[f(y,v)-f(y,-v)]\d v\d y.$$
Again let $$\bar{f}(t,Q,K) = f(t,x,v).$$
Next we use the change of variables $(x,v)\mapsto (Q,K)$, that $v=\sqrt{2H}\sin \chi$ and Lemma~\ref{lem:Jacobian} followed by the fundamental theorem of calculus to obtain
\begin{align*}
&\: \int_0^{x'}\int_0^\infty v[f(y,v)-f(y,-v)]\d v\d y\\
= &\: \int_0^{\Phi(x')}\int_0^{\pi/2} c(K)\sqrt{2K}S(Q,K)[\bar f(Q,K)-\bar f(-Q,K)]\d Q\d {K}\\
&+\int_{\Phi(x')}^\infty\int_{\mathfrak{Q}_K}^{\pi/2} c(K)\sqrt{2K}S(Q,K)[\bar f(Q,K)-\bar f(-Q,K)]\d Q\d {K}\\
=&\:\int_0^{\Phi(x')}\int_0^{\pi/2}\int_{-Q}^{Q} c(K)\sqrt{2K}S(Q,K)\part_Q \bar f(Q',K)\d {Q'}\d Q\d {K}\\
&+\int_{\Phi(x')}^\infty\int_{\mathfrak{Q}_{K}}^{\pi/2}\int_{-Q}^{Q} c(K)\sqrt{2K}S(Q,K)\part_Q \bar f(Q',K)\d {Q'}\d Q\d {K}\\
=: &\: T_1+T_2,
\end{align*}
where we have defined
\begin{itemize}
\item $S(Q,K):=\sin\chi$, and  
\item $\mathfrak{Q}_K$ to be the angle in $(Q,K)$ coordinates corresponding to angle $\arccos\left(\frac{\Phi(x')}{H}\right)$ in $(\chi,H)$ coordinates.
\end{itemize}

Now using Fubini's theorem, we have 
$$T_1=\int_{-\pi/2}^{\pi/2}\int_0^{\Phi(x')}\left(\int^{\pi/2}_{|Q'|} S(Q,K)\d Q\right)c(K)\sqrt{2K}\part_Q \bar f(Q',K)\d {K}\d {Q'}$$
and 
\begin{align*}
T_2&=\int_{-\pi/2}^{\pi/2}\int_{\Phi(x')}^{\mathfrak{H}_{Q'}}\left(\int_{|Q'|}^{\pi/2} S(Q,K)\d Q\right)c(K)\sqrt{2K}\part_Q \bar f(Q',K)\d {K}\d {Q'}\\
&+\int_{-\pi/2}^{\pi/2}\int_{\mathfrak{H}_{Q'}}^\infty\left(\int_{\mathfrak{Q}_{K}}^{\pi/2} S(Q,K)\d Q\right)c(K)\sqrt{2K}\part_Q \bar f(Q',K)\d {K}\d {Q'},
\end{align*}
where $\mathfrak{H}_{Q'}$ is such that $\left(\arccos\left(\frac{\Phi(x')}{\mathfrak{H}_{Q'}}\right),\mathfrak{H}_{Q'}\right)$ in $(\chi,H)$ coordinates gets mapped to $(|Q'|,\mathfrak{H}_{Q'})$ in $(Q,K)$ coordinates  
(such an $\mathfrak{H}_{Q'}$ exists because $\chi=\arccos\left(\frac{\Phi(x')}{H}\right)$ increases as $H$ does and $Q$ is monotone\footnote{Since $a(\chi,H)>0$, we have that $Q$ is monotonically increasing as a function of $\chi$ and vice-versa.} in $\chi$.)

Putting the above together we get,
\begin{align*}
T_1+T_2=&\int_{-\pi/2}^{\pi/2}\int_{0}^{\mathfrak{H}_{Q'}}\left(\int_{|Q'|}^{\pi/2} S(Q,K)\d Q\right)c(K)\sqrt{2K}\part_Q \bar f(Q',K)\d {K}\d {Q'}\\
&+\int_{-\pi/2}^{\pi/2}\int_{\mathfrak{H}_{Q'}}^\infty\left(\int_{\mathfrak{Q}_{K}}^{\pi/2} S(Q,K)\d Q\right)c(K)\sqrt{2K}\part_Q \bar f(Q',K)\d {K}\d {Q'}.
\end{align*}
Now we use \eqref{eq:f-to-g} from Lemma~\ref{lem:f-to-g} and that $\part_Q=\frac{1}{c'({K})t}(Y+\part_{K})$ to get that
\begin{align*}
T_1+T_2=&t^{-1}\int_{-\pi/2}^{\pi/2}\int_{0}^{\mathfrak{H}_{Q'}}\frac{1}{c'({K})}\left(\int_{|Q'|}^{\pi/2} S(Q,K)\d Q\right)c(K)\sqrt{2K} (Y+\part_{K})\part_Q \bar{g}(Q',K)\d {K}\d {Q'}\\
&+t^{-1}\int_{-\pi/2}^{\pi/2}\int_{\mathfrak{H}_{Q'}}^\infty\frac{1}{c'({K})}\left(\int_{\mathfrak{Q}_{K}}^{\pi/2} S(Q,K)\d Q\right)c(K)\sqrt{2K}(Y+\part_{K})\part_Q \bar{g}(Q',K)\d {K}\d {Q'}.
\end{align*}

Next we integrate by parts in $K$. Since $\mathfrak{Q}_{\mathfrak{H}_{Q'}}=|Q'|$, we see that the boundary terms exactly cancel! Hence,
\begin{align*}
&t^{-1}\int_{-\pi/2}^{\pi/2}\int_{0}^{\mathfrak{H}_{Q'}}\frac{1}{c'(K)}\left(\int_{|Q'|}^{\pi/2} S(Q,K)\d Q\right)c(K)\sqrt{2K} \part_{K}\part_Q \bar{g}(Q',K)\d {K}\d {Q'}\\
&+t^{-1}\int_{-\pi/2}^{\pi/2}\int_{\mathfrak{H}_{Q'}}^\infty\frac{1}{c'(K)}\left(\int_{\mathfrak{Q}_{K}}^{\pi/2} S(Q,K)\d Q\right)c(K)\sqrt{2K}\part_{K}\part_Q \bar{g}(Q',K)\d {K}\d {Q'}\\
&\quad=-t^{-1}\int_{-\pi/2}^{\pi/2}\int_{0}^{\mathfrak{H}_{Q'}}\part_K\left(\frac{1}{c'(K)}\left(\int_{|Q'|}^{\pi/2} S(Q,K)\d Q\right)c(K)\sqrt{2K}\right)\part_Q \bar{g}(Q',K)\d {K}\d {Q'}\\
&\quad -t^{-1}\int_{-\pi/2}^{\pi/2}\int_{\mathfrak{H}_{Q'}}^\infty\part_K\left(\frac{1}{c'(K)}\left(\int_{\mathfrak{Q}_{K}}^{\pi/2} S(Q,K)\d Q\right)c(K)\sqrt{2K}\right)\part_Q \bar{g}(Q',K)\d {K}\d {Q'}.
\end{align*}

Since there is no boundary term we can integrate by parts after writing $\part_Q=\frac{1}{c'({K})t}(Y+\part_{K})$ once more. Next note that that $\bar{g}(Q,K)$ is nonzero only for $K \in [c_s,c_{s}^{-1}]$ and that derivatives of $\frac{c(K)}{c'(K)}$ is bounded as $|c'(K)|\geq \delta$ by Lemma~\ref{lem:c-prime-bound}. Futher, $S(Q,K)=\sin\chi$ is smooth as a function of $K$. Thus $$\sum_{\ell\leq 2} \partial_K^\ell \left(\frac{1}{c'(K)}\left(\int_{|Q'|}^{\pi/2} S(Q,K)\d Q\right)c(K)\sqrt{2K}\right)\lesssim 1.$$

By Cauchy--Schwarz in $Q'$ and $K$, we get that
$$T_1+T_2\ls \sum_{\ell\leq 2}\sup_{K}\norm{Y^\ell \bar{g}}_{L^2_Q }.$$
Finally, an application of \eqref{eq:g-bound} from Lemma~\ref{lem:f-to-g} followed by Lemma~\ref{lem:infinity-est} gives us the required bound. \qedhere

\end{proof}
\begin{proof}[Proof of Theorem~\ref{thm:main}]
The proof follows by using Lemma~\ref{lem:phi_t} and combining the estimates from Proposition~\ref{prop:term-x=0} and Proposition~\ref{prop:bulk-term}. 
\end{proof}
\bibliographystyle{plain}
\bibliography{VPL}
\end{document}